\documentclass[10pt]{article}
\usepackage{amssymb,amsthm}
\usepackage[tbtags]{amsmath}
\usepackage{color}
\newtheorem{thm}{Theorem}[section]
\newtheorem{corollary}{Corollary}[section]
\newtheorem{lem}{Lemma}[section]

\theoremstyle{definition}
\newtheorem{de}{Definition}[section]

\theoremstyle{remark}
\newtheorem{rem}{Remark}[section]

\usepackage{color}

\definecolor{c20}{rgb}{0.,0.7,0.}
\definecolor{c30}{rgb}{0.,0.,1.}
\definecolor{c40}{rgb}{1,0.1,0.7}
\definecolor{c50}{rgb}{1,0,0}
\definecolor{c20}{rgb}{0.,0.7,0.}
\definecolor{c30}{rgb}{0.,0.,1.}
\definecolor{c40}{rgb}{1,0.1,0.7}
\definecolor{c50}{rgb}{1,0,0}
\def\IF{\infty}

\newcommand{\EE}[1]{\mathbb{E}\left\{#1\right\}}

\def\iE#1{\textcolor{c20}{#1}}
\def\iE#1{#1}

\newcommand{\kb}[1]{\boldsymbol{#1}}
\newcommand{\vk}[1]{\kb{#1}}

\newcommand{\R}{\mathbb{R}}

\newcommand{\inr}{\in \R}

\newcommand{\BQN}{\begin{eqnarray}}
\newcommand{\EQN}{\end{eqnarray}}
\newcommand{\BQNY}{\begin{eqnarray*}}
\newcommand{\EQNY}{\end{eqnarray*}}

\newcommand{\BT}{\begin{theo}}
\newcommand{\ET}{\end{theo}}
\newcommand{\BK}{\begin{korr}}
\newcommand{\EK}{\end{korr}}

\newcommand{\BEL}{\begin{lem}}
\newcommand{\EEL}{\end{lem}}

\def\pE#1{#1}

\newtheorem{theo}{Theorem}[section]

\newcommand{\COM}[1]{}

\newcommand{\norm}[1]{\lVert #1 \rVert}

\def\wF{{\wwF}}
\def\rE{\rangle}
\def\rEE{\rangle}
\def\nO{\|}
\def\nOO{\|}
\def\kHA{\mathcal{H}_1}
\def\kHB{\mathcal{H}_2}
\def\kHC{\mathcal{H}_{2,+}}
\newcommand{\pk}[1]{\mathbb{P} \left\{ #1 \right \} }

\begin{document}

\centerline{\bf \large
Boundary Non-Crossings of Additive Wiener Fields}

 \bigskip
\centerline{Enkelejd  Hashorva\footnote{Department of Actuarial Science, University of Lausanne, UNIL-Dorigny, 1015 Lausanne, Switzerland,
email:enkelejd.hashorva@unil.ch} and
Yuliya Mishura\footnote{Department of Probability, Statistics and Actuarial Mathematics, National Taras Shevchenko University of Kyiv, 01601 Volodymyrska 64, Kyiv, Ukraine, email: myus@univ.kiev.ua }}

\bigskip
\centerline{\today{}}

{\bf Abstract}:
Let $W_i=\{W_i(t), t\in \R_+\}, i=1,2$ be two Wiener processes and  $W_3=\{W_3(\mathbf{t}), \mathbf{t}\in \R_+^2\}$ be a two-parameter Brownian sheet, all three processes being mutually independent. We derive upper and lower bounds for the  boundary non-crossing probability $$P_f=P\{W_1(t_1)+W_2(t_2)+W_3(\mathbf{t})+h(\mathbf{t})\leq u(\mathbf{t}), \mathbf{t}\in\R_+^2\},$$
where  $h, u: \R_+^2\rightarrow \R_+$ are two measurable functions. We show further that for large trend functions $\gamma f>0$ asymptotically when $\gamma \to \IF$ we have that
$\ln P_{\gamma f}$ is the same as
$\ln P_{\gamma \underline{f}}$ where $\underline{f}$ is the projection of $f$ on some closed convex set of the reproducing kernel Hilbert Space of $W$.
It turns out that our approach is applicable also  for the additive Brownian pillow.



%
\section{INTRODUCTION}\label{s:1}
Calculation of boundary non-crossing probabilities of Gaussian processes is a key topic both of theoretical and  applied probability, see, e.g., 
\cite{Durb92,Wang97,Nov99, Wang2001,MR2009980,BNovikov2005,MR2028621,1103.60040,1137.60023,1079.62047,MR2576883,janssen}
and the references therein. Numerous applications concerned with the evaluation of boundary non-crossing probabilities relate  to mathematical finance, risk theory, queueing theory, statistics, physics among many other fields. Also calculation of boundary non-crossing  probabilities of random fields are considered in various contexts, see e.g., \cite{Pit96, Csaki, Pillow, Somayasa13}. Unlike the previous papers,  we consider in this contribution the general model consisting of three components that include a standard Brownian sheet and two independent Wiener processes. We can not apply the methods proposed for Brownian pillow since they are based \pE{on the fact that it vanishes} on some rectangle. Therefore, we modify essentially the methods from \cite{BiHa1, MR2028621, Pillow} to meet the properties of our model, and in that context some additional conditions are introduced in our main result. The choice of the model is quite natural.  Indeed, on one hand, the model consists of three Gaussian processes that are independent, have continuous trajectories and independent increments, so the model is clear and tractable. On the other hand, arbitrary function\pE{s} defined on the positive \pE{quadrant}, can be decomposed \pE{uniquely} into three components, two of them representing its behavior on the axes and the third component being zero on the axes.

\begin{de} Brownian sheet $\widetilde{W}=\{\widetilde{W}(\mathbf{t}), \mathbf{t}\in \R_+^2\}$ is a Gaussian field with zero mean and covariance function
$$ \EE{\widetilde{W}(\mathbf{t})\widetilde{W}(\mathbf{s}}=(s_1\wedge t_1)(s_2\wedge t_2).$$
\end{de}  Evidently, Brownian sheet is zero on the axes and in what follows we shall consider its continuous modification.

Let $W_i=\{W_i(t), t\in \R_+\}, i=1,2$ be two Wiener processes  and let
$W_3=\{W_3(\mathbf{t}), \mathbf{t}\in \R_+^2\}$ be a Brownian sheet.  For two measurable functions $f,u:\R_+^2\rightarrow \R$
we shall investigate the boundary non-crossing probability
$$P_f=\pk{f(\mathbf{t})+W(\mathbf{t})\leq u(\mathbf{t}),\;\mathbf{t}\in\R_+^2},$$
with  $W$ an additive Wiener field  defined by 
\begin{equation}\label{eqhm-1}
W(\mathbf{t})= W_1(t_1)+W_2(t_2)+W_3(\mathbf{t}), \quad \mathbf{t} \in\R^2_+,
\end{equation}
where we assume that $W_1,W_2,W_3$ are mutually independent.  Clearly, the additive Wiener field $W$ is a centered Gaussian field with covariance function
\begin{equation}\label{eqhm-3}
\EE{W(\mathbf{s})W(\mathbf{t})}=s_1\wedge t_1+s_2\wedge t_2+(s_1\wedge t_1)(s_2\wedge t_2), \quad
\mathbf{s}=(s_1,s_2), \mathbf{t}=(t_1,t_2).
\end{equation}

As it is commonly the case for random fields, also for the additive Wiener field
explicit calculations of boundary non-crossing probabilities \pE{are not available} even for the case that both $f,u$ are constant\pE{s}, see e.g., \cite{Csaki}.  Therefore in our analysis we shall derive upper and lower bounds considering general measurable function\pE{s}  $u$ and function $f$   from the reproducing kernel Hilbert space (RKHS)
 of $W$ denoted by $\kHC$. In order to determine $\kHC$ we need to recall first the corresponding RKHS of $W_1$, $W_2$ and $W_3$.
It is well-known (see e.g., \cite{berlinet})
that the RKHS of the Wiener process $W_1$, denoted by $\kHA$, is characterized as follows
$$\kHA=\Bigl\{h:\R_+\rightarrow \R\big|h(t)=\int_{[0,t]}h'(s)ds,\quad  h'\in L_2(\R_+, \lambda_1) \Bigr\}, $$
with the inner product $\langle h,g\rangle=\int_{\R_+}h'(s)g'(s)ds$ and the corresponding norm $\|h\|^2=\langle h,h\rangle$. The description of RKHS for $W_2$ is evidently the same.
It is also well-known that the RKHS of the Brownian sheet $W_3$, denoted by $\kHB$, is characterized as follows
$$\kHB= \Bigl\{h:\R_+^2\rightarrow \R\big|h(\mathbf{t})=\int_{[0,\mathbf{t}]}h''(\mathbf{s})d\mathbf{s}, \quad h''\in L_2(\R_+^2, \lambda_2)\Bigr\},$$
with the inner product $\langle h,g\rangle=\int_{\R_+^2}h''(\mathbf{s})g''(\mathbf{s})d\mathbf{s}$ and the
corresponding norm $\|h\|^2=\langle h,h\rangle.$  Here the symbols $\lambda_1$ and $\lambda_2$ stand for the Lebesgue measures  in the $\R_+^1$ and in $\R_+^2$, \pE{respectively}. As shown in Lemma \ref{lem:kHC} in Appendix the RKHS corresponding to the covariance function of the additive Wiener field $W$ given in \eqref{eqhm-3} is
\begin{equation}
\kHC=\Bigl\{h:\R_+^2\rightarrow \R\big|h(\mathbf{t})=h_1(t_1)+h_2(t_2)+h_3(\mathbf{t}),\; \text{where}\;  h_i\in \kHA,  i=1,2 \; \text{and}\; h_3\in \kHB \Bigr\}
\end{equation}\label{eqhm-9}
equipped with the inner product
\begin{equation}\label{eqhm-10}
\langle h,g\rangle=\int_{\R_+}h_1'(s)g_1'(s)ds+  \int_{\R_+}h_2'(s)g_2'(s)ds+ \int_{\R_+^2}h''(\mathbf{s})g''(\mathbf{s})d\mathbf{s}
\end{equation}
and the corresponding norm $\|h\|^2=\langle h,h\rangle.$
For simplicity we used the same notation for the norm and the inner product of $\kHA,\kHB$ and $\kHC$. Note that in the case when $h\in\kHB\cap C^2(\R^2)$ we have that  $h''\pE{(u,s)}=\frac{\partial^2 h(u,s)}{\partial u\partial s}$, and it is the motivation for  the notation $h''$.
As in \cite{HMS12}, a direct application of Theorem 1' in \cite{LiKuelbs}
shows that for any $f\in  \kHC $ we have
\newcommand{\Abs}[1]{\Bigl\lvert #1 \Bigr\rvert}
 \BQN\label{eq:00:2b}
\Abs{P_f - P_0} &\le \frac {1 }{\sqrt{2 \pi}} \norm{ f}.
\EQN
Clearly, the above inequality provides a good bound for the approximation rate of $P_f$ by $P_0$ when $\norm{f}$ is small. \pE{Recall that $P_0$
 cannot be calculated explicitly, however it can be determined with a certain accuracy by using simulations.} In case that we want to compare $P_f$ and $P_g$ for $g\in \kHC$ and $g\ge f$, we obtain further (by Theorem 1' in \cite{LiKuelbs})
 that
\def\wF{\underline{f}}
\BQN\label{eq:WL}
\Phi(\alpha - \norm{g}) \le P_{g}\le P_f \le \Phi(\alpha+ \norm{f}),
\EQN
where $\Phi$ is the distribution of an $N(0,1)$ random variable and $\alpha=\Phi^{-1}(P_0) $ is a finite constant. When $f\le 0$, then we can take always $g=0$ above.
If  $f(\mathbf{t}_0)>0$ for some $\mathbf{t}_0$ with non-negative components, then  the last inequalities are useful when $\norm{f}$ is large.
\pE{Indeed, for any  $g\ge f, g\in \kHC$ using \eqref{eq:WL} we obtain as $\gamma \to \IF$}
\BQNY
\ln P_{\gamma f} \ge    \ln
\Phi(\alpha - \gamma \wF ) \ge  -(1+o(1))\frac{\gamma^2}{2}\norm{g}^2, 
\EQNY
hence
\BQN\label{LD}
{\ln P_{\gamma f} \ge -(1+o(1))\frac{\gamma^2}{2}\norm{\wF}^2, \quad \gamma \to \IF},
\EQN
where $\wF$ (which is unique and exists) satisfies
\BQN \label{OP}
\min_{ g,f \in { \kHC}, g \ge f}   \norm{g}= \norm{\wF}>0.
\EQN

\COM{\pE{Indeed, taking $g=\wF \in \kHC$ where $\wF\ge f$
and further
\BQN \label{OP}
\min_{ g,f \in { \kHC}, g \ge f}   \norm{g}= \norm{\wF},
\EQN
then from \eqref{eq:WL} we obtain}
\BQN\label{LD}
{\ln P_{\gamma f} \ge    \ln
\Phi(\alpha - \gamma \wF ) \ge  -(1+o(1))\frac{\gamma^2}{2}\norm{\wF}^2, \quad \gamma \to \IF}.
\EQN
}
\pE{In Section} 2 we  identify  $\underline{f}$ with  the projection of $f$ on a closed convex set of $\kHC$, and moreover
we show that
\BQN\label{LD2}
{\ln P_{\gamma f} \sim  \ln P_{\gamma \underline{f}} \sim - \frac{\gamma^2}{2}\norm{\underline{f}}^2, \quad \gamma \to \IF}.
\EQN

Our results in this paper are of both theoretical and practical interest. Furthermore, our approach can be applied when dealing instead of the additive Wiener sheet $W$
 with the linear combinations of $W_1,W_2,W_3$. Additionally, our approach is applicable also for the evaluations of boundary non-crossing probabilities of
the additive  Brownian pillow, i.e., when  $W_1,W_2$ are independent Brownian bridges and $W_3$ is a Brownian pillow. For the later case our results are more general than
those in \cite{Pillow}.

Organization of the paper is as follows: We continue below with preliminaries followed then by a section containing the main result. In Appendix we present three
technical lemmas. Lemma 3 contains It\^{o} formula for the product of two fields in the plane, one of them being Brownian shett and another one having bounded variation. It is used in the proof of the main result, Theorem 1. Lemma 4 states that the RKHS of $W$ is determined uniquely and Lemma 5 describes the asymptotic behavior of $h''$ for $h$ from the closed convex subset of $\kHC$ that is used for projection.




\section{Preliminaries}
Recall that \pE{in this paper} bold letters are reserved for vectors, so we shall write for instance
$\mathbf{t}=(t_1, t_2)\in\R^2_+$ and   $\lambda_1$ and $\lambda_2$ denote the Lebesgue measures on $\R_+$ and $\R_+^2$, respectively
whereas  $ds$ and $d\mathbf{s}$ mean integration with respect to these measures.
\COM{We consider three mutually independent processes: two Wiener processes  $W_i=\{W_i(t), t\in \R_+\}, i=1,2$
and a Brownian sheet $W_3=\{W_3(\mathbf{t}), \mathbf{t}\in \R_+^2\}$. Define a new  field $W$ by
\begin{equation}\label{eqhm-1}
W(\mathbf{t})= \{W_1(t_1)+W_2(t_2)+W_3(\mathbf{t}), \mathbf{t} \in\R^2_+\}
\end{equation}
Evidently, $W$ is the centered Gaussian field.}


\subsection{Expansion of one-parameter functions}

The results of  this subsection \pE{were formulated} in a modified form in e.g., in \cite{BiHa1, janssen, Pillow}.  However we shall introduce some modifications (re-writing for instance $V_1$ below) which are important for the two-parameter case. From the derivations below it will become clear how to
obtain expansion of multiparameter functions \pE{of} two components, one of which is  the ``analog of the smallest concave majorant''  and the other one is a negative function. \pE{Specifically,}  when studying the boundary
crossing probabilities of the Wiener process with a deterministic trend $h\in \kHA$, then it has been shown (see \cite{1137.60023}), that
the smallest concave majorant of $h$ solves \eqref{OP} and determines the large deviation asymptotics of this probability. Moreover, as shown in \cite{janssen}
 the smallest concave majorant of $h$, which we denote by $\underline{h}$,
 can be written analytically as the unique projection of $h$ on the  closed convex set
$$V_1=\{h\in \kHA \big|\,h'(s) \; \text{is a non-increasing function} \}$$
i.e., $\underline{h}= Pr_{V_1}h$. Here we write $Pr_{A}h$ for the projection of $h$ on some closed set $A$ also for
other Hilbert spaces considered below.


\begin{lem}\label{lemma 3.1} Let
$\widetilde{V}_1=\{h\in \kHA  \big|\,\langle h,f\rE \leq 0 \;\text{for any}\; f\in V_1\} $ be the polar cone of  $V_1$.
\begin{itemize}
\item[(i)] If $h\in \widetilde{V}_1$,  then $h\leq 0$.
\item[(ii)] We have $\langle Pr_{V_1}h, Pr_{\widetilde{V}_1}h\rangle=0$ and further
\begin{equation}\label{eqhm-6}
h=Pr_{V_1}h+Pr_{\widetilde{V}_1}h.
\end{equation}
\item[(iii)] If $h=h_1+h_2$, $h_1\in V_1$, $h_2\in \widetilde{V}_1$ and $\langle h_1,  h_2\rangle=0$, then $h_1=Pr_{V_1}h$ and $h_2=Pr_{\widetilde{V}_1}h$.
\item[(iv)] The unique solution of the minimization problem $\min_{g\geq h, g\in \kHA }\norm{g}$ is $\underline{h}=Pr_{V_1}h$.

    \end{itemize}
\end{lem}
\begin{proof} In the following for a given real-valued function $\varphi$
we denote its one-parameter increment $\Delta_s^1\varphi(t)=\varphi(t)-\varphi(s)\;,0\leq s\leq t<\infty$. With this   notation we
can re-write $V_1$ as $$V_1=\{h\in \kHA  \big|\,\Delta_s^1 h'(t)\leq 0\;, 0\leq s\leq t<\infty\}.$$
Let $h\in \widetilde{V}_1$ and define  $A=\{s\inr_+: h(s)>0\}$. Fix $T>0$ and consider the function $v$ such that
$$v'(s)=\int_{[s,T]} h(u)\pE{1_{u\in A}} du1_{s\leq T}.$$
For any $0\leq s\leq t< \infty$ we have $\Delta_s^1v'(t)=-\int_{[s\wedge T,t\wedge T]}h(u)\pE{1_{u\in A}} du\leq 0$ and further
\BQNY
\int_{\R_+}|v'(s)^2|ds &=&\int_{[0,T]}\Big(\int_{[s,T]}h(u) \pE{1_{u\in A}} du\Big)^2ds\\
&\leq & T^2 \int_{[0,T]}h^2(u)du\\
&=&T^2 \int_{[0,T]}\Big(\int_{[0,u]}h'(s)ds\Big)^2du\\
&\leq & T^4 \int_{\R_+}(h'(s))^2ds\\
&<& \infty.
\EQNY
Consequently, $v'\in L_2(\R_+, \lambda_1), v(s)=\int_{[0,s]}v'(u)du\in \kHA $ \pE{and further} $v\in V_1$. Therefore,
\BQN \label{eqhm-8}
 0&\geq &\langle h, v\rangle\notag\\
 &=&\int_{\R_+ }h'(s)v'(s)ds\\
 &=&\int_{[0,T]}h'(s)\int_{[s,T]}h(u)\pE{1_{u\in A}}duds\notag\\
 &=&\int_{[0,T]}h(u)\pE{1_{u\in A}}\int_{[0,u]}h'(s)dsdu \notag\\
 &=&\int_{[0,T]} h^2(u)\pE{1_{u\in A}}du
 \EQN
 \def\unH{\underline{h}}
implying that $\pE{1_{u\in A}}=0$ a.e. $\lambda_1$, in other words, $h(u)\leq 0$ a.e. $\lambda_1$. However, $h$ is a continuous function and therefore $h(u)\leq 0$ for any $u$.\\
Statements $(ii)$ and $(iii)$ follow immediately from \cite{janssen} and are valid for any Hilbert space.\\
$(iv)$  Write $$ f= h+ \varphi= \underline{h}+ \varphi+ h- \underline{h} =\underline{h}+ \varphi+ Pr_{\widetilde{V}_1} h$$
and suppose that $f\in \kHA$ and $\varphi\ge 0$. Note that for any function $g\in V_1$ its derivative $g'$
 is non-increasing therefore  \pE{$g'$ } is  non-negative and \pE{$\lim_{t\to \IF} g'(t)=0$.} Since $\varphi\ge 0$, then for any sequence
  $t_n\rightarrow \infty$ \pE{we have}
  $$\lim_{n\rightarrow \infty}\varphi(t_n)\unH'(t_n)\geq 0,$$
which implies \BQN\label{eqhm-4}
\langle \unH,\varphi\rE &=&\int_{\R_+}\unH '(u)\varphi'(u)du \notag\\
&=&\lim_{n\rightarrow \infty}\int_{[0,t_n]}\unH'(u)\varphi'(u)du\notag\\
&=&\lim_{n\rightarrow \infty}\Big(\varphi(t_n)\unH' (t_n)
-\int_{[0,t_n]}\varphi(u)d(\unH'(u))\Big)\notag\\
&\ge &\lim_{n\rightarrow \infty}\Big(-\int_{[0,t_n]}\varphi(u)d(\unH'(u))\Big)\notag\\
&\geq& 0.
\EQN
\COM{
Now, consider the sequence $g_n(u)=g(u)1_{u\in[0,n]},n\ge 1$. Since $g_n'(u)=g'(u)1_{u\in[0,n]}$  a.e. $\lambda_1$ then \eqref{eqhm-4} entails
$$\int_{\R_+} \unH'(u)g_n'(u)du\geq 0.$$ Therefore
\BQNY \label{eqhm-5}
 \langle \unH ,g\rE& =&\int_{\R_+}\unH'(u)g'(u)du\\
&=&\lim_{n\rightarrow \infty}\int_{[0,n]}\unH '(u)g'(u)du\\
&=&\lim_{n\rightarrow \infty}\int_{[0,n]}\unH '(u)g_n'(u)du\\
&\geq & 0,
\EQNY
}
Consequently,
\BQNY
 \|f\nO ^2=\|h+\varphi\nO ^2 &=& \norm{\underline{h}+ \varphi+ Pr_{\widetilde{V}_1} h}^2\\
 &=& \norm{\underline{h}}^2+ 2\langle  \underline{h}, \varphi \rangle
+ 2\langle  \underline{h}, Pr_{\widetilde{V}_1} h \rangle  + \norm{ \varphi+ Pr_{\widetilde{V}_1} h}^2\\
 &=& \norm{\underline{h}}^2+ 2\langle  \underline{h}, \varphi \rangle + \norm{ \varphi+ Pr_{\widetilde{V}_1} h}^2\\
&\geq& \|\unH\nO ^2
\EQNY
establishing the proof.
\end{proof}

\subsection{Expansion of two-parameter functions}
 For some given measurable function $\varphi:\R_+^2\rightarrow \R$ we define
 $$\Delta_\mathbf{s}\varphi(\mathbf{t})=\varphi(\mathbf{t})-\varphi(s_1, t_2)-\varphi(t_1,s_2)+\varphi(\mathbf{s}),$$
$$\Delta^1_\mathbf{s}\varphi(t_1,s_2)=\varphi(t_1,s_2)-\varphi(\mathbf{s}),\quad \Delta^2_\mathbf{s}\varphi(s_1,t_2)=\varphi(s_1,t_2)-\varphi(\mathbf{s}).$$
In our notation $\mathbf{s}=(s_1,s_2)\leq \mathbf{t}=(t_1,t_2)$ means that $s_1\leq t_1$ and $s_2\leq t_2$.
Define the closed convex set
$$V_2=\{h \in \kHB \big|\Delta_\mathbf{s}h''(\mathbf{t})\geq 0,\;\Delta^1_\mathbf{s}h''(t_1,s_2)\leq 0,\;\Delta^2_\mathbf{s}h''(s_1,t_2)\leq 0\;\text{for any}\; \mathbf{s}\leq\mathbf{t}\; \text{and}\;\mathbf{t}\in \R^2_+ \}$$
and let $\widetilde{V}_2$ be the polar cone of $V_2$, namely
$$\widetilde{V}_2=\{h \in \kHB \big|\langle h, v\rEE  \leq 0 \;\text{for any}\; v\in V_2\}.$$
Below we derive the expansion for two-parameter functions. Since the results are very similar to the previous lemma,
we shall prove only those statements that differ in details from Lemma \ref{lemma 3.1}.
\begin{lem}\label{lemma 3.2}\begin{itemize}
\item[(i)] If  $h\in \widetilde{V}_2$, then $h\leq 0$.
\item[(ii)] We have $\langle Pr_{V_2}h, Pr_{\widetilde{V}_2}h\rEE  =0$ and \begin{equation*}\label{eqhm-7}
h=Pr_{V_2}h+Pr_{\widetilde{V}_2}h.
\end{equation*}
\item[(iii)] If $h=h_1+h_2$, $h_1\in V_2$, $h_2\in \widetilde{V}_2$ and $\langle h_1,  h_2\rEE  =0$, then $h_1=Pr_{V_2}h$ and $h_2=Pr_{\widetilde{V}_2}h$.
\item[(iv)] The unique solution of the minimization problem $\min_{g\geq h, g\in \kHB }\|g\nOO $ is $\underline{h}=Pr_{V_2}h$.
    \end{itemize}
\end{lem}
\begin{proof} We prove only statement $(i)$. Denote $\mathbf{T}=(T,T), T>0$ and consider the function $v$ with $$v''(\mathbf{s})=\int_{[\mathbf{s},\mathbf{T}]}h(\mathbf{u})\pE{ 1_{\mathbf{u}\in \mathbf{A}}}d\mathbf{u}1_{\mathbf{s}\leq \mathbf{T}},$$
where $\mathbf{A}=\{\mathbf{s}\in\R_+^2\big|h(\mathbf{s})\geq 0\}.$ Then for any $\mathbf{0}\leq \mathbf{s}\leq \mathbf{t}$
\BQNY
\Delta_\mathbf{s}^1v''(t_1,s_2)&=&-\int_{[\mathbf{s}\wedge\mathbf{T},(t_1\wedge T,T)]}h(\mathbf{u}) \pE{ 1_{\mathbf{u}\in \mathbf{A}}} d\mathbf{u}\leq 0,\\\Delta_\mathbf{s}^1v''(s_1,t_2)&=&-\int_{[\mathbf{s}\wedge\mathbf{T},(T,t_2\wedge T)]}h(\mathbf{u})\pE{ 1_{\mathbf{u}\in \mathbf{A}}}d\mathbf{u}\leq 0,\\ \Delta^2_\mathbf{s}v''(\mathbf{t})
&=&\int_{[\mathbf{s}\wedge\mathbf{T},\mathbf{t}\wedge\mathbf{T}]}h(\mathbf{u})\pE{ 1_{\mathbf{u}\in \mathbf{A}}} d\mathbf{u}\geq 0.
\EQNY
Furthermore,
\BQNY
\int_{\R_+^2}|v''(\mathbf{s})^2|d\mathbf{s}&=&\int_{[\mathbf{0},\mathbf{T}]}\Big(\int_{[\mathbf{s},\mathbf{T}]}
h(\mathbf{u})\pE{ 1_{\mathbf{u}\in \mathbf{A}}} d\mathbf{u}\Big)^2d\mathbf{s}\\
&\leq & T^4 \int_{[\mathbf{0},\mathbf{T}]}h^2(\mathbf{u})d\mathbf{u}\\
&=&T^4 \int_{[\mathbf{0},\mathbf{T}]}\Big(\int_{[\mathbf{0},\mathbf{u}]}h''(\mathbf{s})d\mathbf{s}\Big)^2d\mathbf{u}\\
&\leq &T^8 \int_{\R_+^2}(h''(\mathbf{s}))^2d\mathbf{s}\\
&<&\infty.
\EQNY
Consequently,
$$v''\in L_2(\R_+^2, \lambda_2), \quad v(\mathbf{s})=\int_{[\mathbf{0},\mathbf{s}]}v''(\mathbf{u})d\mathbf{u}\in \kHB $$
  and
further $v\in V_2$. Similarly to \eqref{eqhm-8} we conclude that  $\pE{ 1_{\mathbf{u}\in \mathbf{A}}}=0$ a.e. $\lambda_2$. Other details follow as in the proof of Lemma \ref{lemma 3.1}.
\end{proof}

Since we are going to work with functions $f$ in $\kHC$ we need to consider the projection of such $f$ on a particular closed convex set.
In the following we shall write $f=f_1+f_2+f_3$ meaning that $f(\vk{t})=f_1(t_1)+ f_2(t_2) + f_3(\vk{t})$ where $f_1,f_2 \in \kHA$
and $f_3 \in \kHB$. Note in passing that this decomposition is unique for any $f\in \kHC$.
Define the closed convex set
$$V_{2,+}=\{h=h_1+h_2+ h_3 \in \kHC \big|  h_1,h_2 \in V_1, h_3 \in V_2\}$$
and let $\widetilde{V_{2,+}}$ be the polar cone of $V_{2,+}$ given by
$$\widetilde{V_{2,+}}=\{h \in \kHC \big|\langle h, v\rEE  \leq 0 \;\text{for any}\; v\in V_{2,+}\},$$
with inner product from \eqref{eqhm-10}. It follows that for any  $h=h_1+h_2+ h_3\in \widetilde{V}_2$ we have  $h_i\leq 0,i=1,2$
and $h_3\le 0$. Furthermore, $\langle Pr_{V_{2,+}}h, Pr_{\widetilde{V_{2,+}}}h\rEE  =0$ and
\BQN\label{projV3}
h=Pr_{V_{2,+}}h+Pr_{\widetilde{V_{2,+}}}h.
\EQN
Analogous to Lemma \ref{lemma 3.2} we also have
that for  $h=f+g$, $f\in V_{2,+}$, $g\in \widetilde{V_{2,+}}$ such that $\langle f,  g\rEE  =0$, then $f=Pr_{V_{2,+}}h$ and $g=Pr_{\widetilde{V_{2,+}}}h$.
Moreover, the unique solution of \eqref{OP} 
 is
 \BQN \label{projV3B}
 \underline{h}=Pr_{V_{2,+}}h =Pr_{V_1}h_1+ Pr_{V_1}h_2+ Pr_{V_2}h_3.
 \EQN

\section{Main Result}
Consider two measurable two-parameter functions $f,u:\R_+^2\rightarrow \R$. Suppose that $f(\mathbf{0})=0$ and set
\BQNY
f_1(t_1):=f(t_1,0),\;f_2(t_2):=f(0,t_2),
\;f_3(\mathbf{t}):=f(\mathbf{t})-f(t_1,0)-f(0,t_2)),
\EQNY
hence we can write $f(\mathbf{t})=f(t_1,0)+f(0,t_2)+(f(\mathbf{t})-f(t_1,0)-f(0,t_2)) $.
Let   $f_i\in \kHA ,\;i=1,2$ and $f_3\in \kHB$. Recall their representations  $f_i(t)=\int_{[0,t]}f'_i(s)ds,\quad  f'_i\in L_2(\R_+, \lambda_1), i=1,2,$ and $f_3(\mathbf{t})=\int_{[0,\mathbf{t}]}f''_3(\mathbf{s})d\mathbf{s}, \quad f''_3\in L_2(\R_+^2, \lambda_2)$. We shall estimate the boundary non-crossing probability
\BQNY
 P_f=\pk{f(\mathbf{t})+W(\mathbf{t})\leq u(\mathbf{t}),\;\mathbf{t}\in\R_+^2}.
 \EQNY
 In  the following we set  $\underline{f_i}= Pr_{V_1}f_i,i=1,2$ and $\underline{f_3}= Pr_{V_2} f, \underline{f}= Pr_{V_{2,+}}f$ and define
 $$\pE{ \underline{f_{13}}(t)=  \underline{f_1}'(t)-\underline{f_3}''(t,0), \quad
\underline{f_{23}}(t) =  \underline{f_2}'(t)-\underline{f_3}''(0,t).}$$
  We state next our main result:

\begin{thm} \label{Thn1} Let the following conditions hold:
\begin{itemize}
 \item[(i)] both functions $\underline{f_{13}}(t)$ and $\underline{f_{23}}(t)$ are non-negative and non-increasing in their arguments;

\item[\pE{(ii)}]
\BQN\label{conA1}\lim_{t \rightarrow\infty}u(t,0)\underline{f_{13}}(t) &=&\lim_{t \rightarrow\infty}u(0,t)\underline{f_{23}}(t)=0, \quad
\lim_{t_1, t_2 \to \infty}u(\mathbf{t})\underline{f_3}''(\mathbf{t})=0,
\\
\iE{\lim_{x \rightarrow\infty}\int_{[0,x]} u(x,t)d_t(\underline{f_3}''(x,t))}&=&
\iE{\lim_{x \rightarrow\infty}
\int_{[0,x]} u(s,x)d_s(\underline{f_3}''(s,x))}
=0. \label{conA2}
\EQN
\end{itemize}
Then we have
\begin{equation*}\begin{gathered}
P_f\leq P_{ f-\underline{f} }
\exp\biggl (  - \int_{\R_+}u(t,0)d \underline{f_{13}}(t)
-\int_{\R_+}u(0,t)d \underline{f_{23}}(t) +\int_{\R_+^2}u(\mathbf{t})d \underline{f_3}''(\mathbf{t})
-\frac12\|\underline{f}\nOO ^2\biggr).
\end{gathered}
\end{equation*}
\end{thm}
\begin{rem} Note that $f$ starts from zero therefore $f$ can not be a constant unless $f\equiv 0$ but this case is trivial.
\end{rem}
\begin{rem} Condition $(ii)$ of the theorem means that asymptotically the shifts and their derivatives are negligible  in comparison with function $u$. It is the generalization of the corresponding conditions for the Brownian bridge and Brownian pillow that are defined on a compact sets so that the corresponding condition holds automatically.
\end{rem}
\begin{proof}
Denote by $\widetilde{P}$ a probability measure that is defined via its Radon-Nikodym derivative
\begin{equation*}\begin{gathered} \frac{dP}{d\widetilde{P}}=\prod_{i=1,2}\exp\Big(-\frac12\|f_i\nO ^2+\int_{\R_+}f_i'(t)dW_i^0(t)\Big)
\exp\Big(-\frac12\|f_3\nO ^2+\int_{\R_+^2}f_3''(\mathbf{t})dW_3^0(\mathbf{t})\Big).
\end{gathered}
\end{equation*}
According to Cameron-Martin-Girsanov theorem,  $W_i^0(t)=W_i(t) +\int_{[0,t]}f_i'(s)ds,\;i=1,2$ are \pE{independent} Wiener processes and $W_3^0(\mathbf{t})=W_3(\mathbf{t}) +\int_{[\mathbf{0},\mathbf{t}]}f_3''(\mathbf{s})d\mathbf{s}$ is a Brownian sheet w.r.t. the measure $\widetilde{P}$ \pE{being further independent of $W_1^0, W_2^0$.} Denote
$1_u\{X\}=1\{X(\mathbf{t})\leq u(\mathbf{t}),\;\mathbf{t}\in\R_+^2\}$ and $$W^0(\mathbf{t})=W_1^0(t_1)+ W_2^0(t_2)+W^0_3(\mathbf{t}).$$
Note that $ \norm{f}^2= \norm{f_1}^2+\norm{f_2}^2+ \norm{f_3}^2$, \pE{hence
using further} \eqref{projV3} and \eqref{projV3B} we obtain
\BQNY
\lefteqn{P_f}\\
 &=&\EE{ 1_u\Big(\sum_{i=1,2}(W_i(t)+f_i(t))+W_3(\mathbf{t})+f_3(\mathbf{t})\Big)}\\
&=&\mathbb{E}_{\widetilde{P}}\Biggl( \frac{dP}{d\widetilde{P}}1_u\Big(W^0(\mathbf{t})\Big)\Biggr)\\
&=&
\exp\Big(-\frac12 \norm{f}^2 \Big) \EE{ \exp\Big(\int_{\R_+}f_1'(t)dW_1^0(t)
+\int_{\R_+}f_2'(t)dW_2^0(t)+\int_{\R_+^2}f_3''(\mathbf{t})dW_3^0(\mathbf{t})\Big)1_u\Big(W^0(\mathbf{t})\Big)}\\
&=&\exp\Big(-\frac12 \norm{\underline{f}}^2 \Big)\\
&&\times \mathbb{E} \Biggl\{\prod_{i=1,2}\exp\Big(-\frac12\|Pr_{\widetilde{V}_1}{f_i} \nO ^2
+\int_{\R_+} Pr_{\widetilde{V}_1}f_i'(t)dW_i^0(t)\Big)\exp\Big(-\frac12\|Pr_{\widetilde{V_2}}{f_3} \nOO ^2
+\int_{\R_+^2}Pr_{\widetilde{V_2}}{f_3}''(\mathbf{t})dW_2^0(\mathbf{t})
\Big)\\
&&\times \exp\Big(\sum_{i=1,2}\int_{\R_+ } \underline{f_i}'(t)dW_i^0( {t})+\int_{\R_+^2} \underline{f_3}''(\mathbf{t}) dW_2^0(\mathbf{t})\Big)1_u\Big(W^0(\mathbf{t})\Big)\Biggr\}.
\EQNY
Now we only need to re-write $$\sum_{i=1,2}\int_{\R_+ } \underline{f_i}'(t)dW_i^0( {t})+\int_{\R_+^2} \underline{f_3}''(t)dW_3^0(\mathbf{t})=\sum_{i=1,2}\int_{\R_+ } \underline{f_i}'(t)dW_i^0( {t})+\int_{\R_+^2} \underline{f_3}''(t)dW^0(\mathbf{t}).$$
In order to re-write $\int_{\R_+ } \underline{f_1}'(t)dW_1^0( {t})$, we mention that in this integral $dW_1^0( {t})=d_1W_1^0( {t})=d_1(W^0(t,0))$, therefore on the indicator $1_u\{\sum_{i=1,2} W_i^0(t)+W_3^0(\mathbf{t})\}=1_u\{W^0(\mathbf{t})\}$
under conditions of the theorem we have the relations
\begin{equation}\begin{gathered}\label{eq3.2n}
\int_{\R_+}\underline{f_1}' (t)dW_1^0( {t})=\lim_{n\rightarrow \infty}\int_{[0,n]}\underline{f_1}' (t)dW_1^0( {t})\\
=\lim_{n\rightarrow \infty}\Big(\underline{f_1}' (n)W^0(n,0)+\int_{[0,n]}W^0(t,0)d(-\underline{f_1}' )(t)\Big).
\end{gathered}
\end{equation}
Similarly, \begin{equation}\label{eq3.3n}\int_{\R_+ }\underline{f_2}'(t)dW_2^0( {t})
=\lim_{n\rightarrow \infty}\Big(\underline{f_2}' (n)W^0(0,n)+\int_{[0,n]}W^0(0,t)d(-\underline{f_2}' )(t)\Big).\end{equation}
Further,  by Lemma \ref{lemma5.1}
\begin{equation}\begin{gathered}\label{3.5n}
\int_{\R_+^2 }\underline{f_3}''(\mathbf{t})dW^0( {\mathbf{t}})=\lim_{n\rightarrow \infty}
\Big(\underline{f_3}''(\mathbf{n})W^0( {\mathbf{n}})-\underline{f_3}''(n,0)W^0(n,0)-\underline{f_3}''(0,n)W^0( 0,n)\\+\int_{[\mathbf{0},\mathbf{n}]}W^0(\mathbf{t})d  \underline{f_3}''(\mathbf{t})+\int_{[0,n]}W^0(s,n)d_s(-\underline{f_3}''(s,n))
+\int_{[0,n]}W^0(n,t)d_t(-\underline{f_3}''(n,t))\\+\int_{[0,n]}W^0(s,0)d_s(\underline{f_3}''(s,0))
+\int_{[0,n]}W^0(0,t)d_t(\underline{f_3}''(0,t))\Big).
\end{gathered}\end{equation}

Combining \eqref{eq3.2n}--\eqref{3.5n}  and using conditions \pE{(i)-(ii)}, we get that  on the same indicator

\begin{equation}\begin{gathered}\label{3.6n}\sum_{i=1,2}\int_{\R_+ } \underline{f_i}'(t)dW_i^0( {t})+\int_{\R_+^2} \underline{f_3}''(t)dW^0(\mathbf{t})\leq \lim_{n\rightarrow \infty}\Big(\underline{f_3}''(\mathbf{n})u( {\mathbf{n}})+(\underline{f_1}' (n)-\underline{f_3}''(n,0))u(n,0)\\+(\underline{f_2}' (n)-\underline{f_3}''(0,n))u(0,n)+\int_{[\mathbf{0},\mathbf{n}]}u(\mathbf{t})d  \underline{f_3}''(\mathbf{t})+\int_{[0,n]}u(s,n)d_s(-\underline{f_3}''(s,n))
+\int_{[0,n]}u(n,t)d_t(-\underline{f_3}''(n,t))\\+\int_{[0,n]}u(s,0)d_s(\underline{f_3}''(s,0)-\underline{f_1}'(s))
+\int_{[0,n]}u(0,t)d_t(\underline{f_3}''(0,t)-\underline{f_2}'(t)\Big)\\\leq \int_{\R^2_+}u(\mathbf{t})d  \underline{f_3}''(\mathbf{t})+\int_{\R_+}u(s,0)d_s(\underline{f_3}''(s,0)-\underline{f_1}'(s))
+\int_{\R_+}u(0,t)d_t(\underline{f_3}''(0,t)-\underline{f_2}'(t)).
\end{gathered}\end{equation}Further conclusions are similar to \cite{BiHa1}.
\end{proof}

The above theorem applied for $u(s,t)=u>0, s,t\ge 0$ combined with \eqref{LD} implies the following result.
\begin{corollary} If $f\in \kHC$ is such that $f(\mathbf{t}_0)>0$ for some $\mathbf{t}_0$ with non-negative components, then \eqref{LD2} holds.
\end{corollary}

{\bf Remarks}:
a) If $u$ is bounded, then according to Lemma \ref{lem5}, conditions \pE{$(i)-(ii)$} are  satisfied. \\
b) Our results can be generalized to higher dimensions. We only mention that in the   case of $n$-parameter functions  we have to define similarly all the differences $\Delta^k_{\mathbf{s}} f(\mathbf{t}),\;1\leq k\leq n$ and the space $$V_n=\{h \in \mathcal{H}_n^2\big|(-1)^k\Delta^k_\mathbf{s}h(\mathbf{t})\geq 0,\text{for any}\; \mathbf{s}\leq\mathbf{t},\;1\leq k\leq n\}.$$

c) \iE{The case of linear combinations of $W_i$'s can be treated with some obvious modifications. }

d) \iE{Consider  the additive Brownian pillow
$$
B(t_1,t_2)=B_1(t_1)+B_2(t_2)+ B_3(t_1,t_2), \quad t_1,t_2 \in [0,1],
$$
which is constructed similarly to the additive Wiener field; here $B_1,B_2$ are two independent Brownian bridges and $B_3$ is a Brownian pillow being further independent of $B_1, B_2$. The RKHS of $B,B_1,B_3$ are almost the same \pE{as} $W,W_1,W_3$ with the only differences that the corresponding functions are defined on $[0,1]^2$ or $ [0,1]$ and the functions are zero on the boundaries of these intervals. The closed convex spaces $V_1,V_2$ and $V_3$ are then defined similarly as in Section 2, and
  thus all the results above hold for the additive Brownian pillow by simply changing the conditions for $f$ and $u$ accordingly.
  Note that compared to \cite{Pillow} we do not need to put restrictions on $\underline{f}$.  Thus the results obtained by our approach here are more general}.

\section {Appendix}
Let $A\in \kHB $ be a two-parameter non-random function. If $A \in \widetilde{V_2}$, then  $A$ is non-increasing as  function of any one-parameter variable and non-decreasing as a function of two variables. Then for \pE{the} additive Wiener field $W=\{W(\mathbf{t})=W_1(t_1)+W_2(t_2)+W(\mathbf{t}), \mathbf{t}\in \R_+^2\}$ and for any $\mathbf{T}=(T,T)$ there exist two integrals of the first kind (according to the classification from the papers \cite{Cairoli,Wong} and \cite{Wong-1}), $\int_{[\mathbf{0},\mathbf{T}]}A(\mathbf{u})dW(\mathbf{u})$ that is standard integral of non-random function with respect to a Gaussian process, or It\^{o} integral, which is the same in this case because $$\int_{[\mathbf{0},\mathbf{T}]}A(\mathbf{u})dW(\mathbf{u})=\int_{[\mathbf{0},\mathbf{T}]}A(\mathbf{u})dW_3(\mathbf{u}),$$ and $\int_{[\mathbf{0},\mathbf{T}]}W(\mathbf{u})dA(\mathbf{u})$ that is the Riemann-Stieltjes integral. We argue only for the first integral. Indeed, such function $A$ achieves its maximal value at $\mathbf{0}.$ Therefore $\int_{[\mathbf{0},\mathbf{T}]}A^2(\mathbf{s})d\mathbf{s}\leq A(\mathbf{0})T^2$ which implies
that the  integral $\int_{[\mathbf{0},\mathbf{T}]}A(\mathbf{u})dW_3(\mathbf{u})$
 is correctly defined as It\^{o} integral. Moreover, denote the increments
 $$\Delta^1_{ik,n}X=\Delta^1_{\big(\frac{T(i-1)}{n}, \frac{T(k-1)}{n}\big)}X\Big(\frac{Ti }{n}, \frac{T(k-1)}{n}\Big)$$
  and
  $$\Delta^2_{ik,n}X=\Delta^1_{\big(\frac{T(i-1)}{n}, \frac{T(k-1)}{n}\big)}X\Big(\frac{T(i-1)}{n}, \frac{Tk}{n}\Big),$$
  $X=A,W$.    Then there exist two integrals of the second kind
   $$\int_{[\mathbf{0},\mathbf{T}]}d_iA(\mathbf{u})d_jW(\mathbf{u}),\quad i=1,2, j=3-i,$$
that are defined as the limits in probability of integral sums where for example, $$\int_{[\mathbf{0},\mathbf{T}]}d_1A(\mathbf{u})d_2W(\mathbf{u})=\lim_{n\rightarrow\infty}\sum_{1\leq i,k\leq n}\Delta^1_{ik,n}A\Delta^2_{ik,n}W.$$
\begin{lem}\label{lemma5.1}
Let $A\in \widetilde{V}_2$ be a two-parameter non-random function and let $W=\{W(\mathbf{t}), \mathbf{t}\in \R_+^2\}$ be an additive Wiener field. Then for any $\mathbf{T}=(T,T)$ we have the following version of integration-by-parts formula:
\BQNY
\int_{[\mathbf{0},\mathbf{T}]}A(\mathbf{s})dW(\mathbf{s})
&=&
A(\mathbf{T})W(\mathbf{T})-A(T,0)W(T,0)-A(0,T)W(0,T)+\int_{[\mathbf{0},\mathbf{T}]}W(\mathbf{s})dA(\mathbf{s})\\
&&+\int_{[0,T]}W(s,T)d_s(-A(s,T))
+\int_{[0,T]}W(T,t)d_t(-A(T,t))\\&&+\int_{[0,T]}W_1(s)d_s(A(s,0))+\int_{[0,T]}W_2(t)d_s(A(0,t)).
\EQNY
\end{lem}
\begin{proof} The standard one-parameter It\^{o} formula yields
 $$\int_{[0,T]}A(s,T)d_sW(s,T)=A(\mathbf{T})W(\mathbf{T})-A(0,T)W(0,T)-\int_{[0,T]}W(s,T)d_sA(s,T).$$
Using further the generalized two-parameter It\^{o} formula (see e.g., \cite{mishura}), we get
 $$\int_{[0,T]}A(s,T)d_sW(s,T)=\int_{[0,T]}A(s,0)dW_1(s)+\int_{[\mathbf{0},\mathbf{T}]}A(\mathbf{s})dW(\mathbf{s})
 +\int_{[\mathbf{0},\mathbf{T}]}d_1W(\mathbf{t})d_2A(\mathbf{t}),$$
 and similarly $$\int_{[0,T]}W(T,t)d_tA(T,t)=\int_{[0,T]}W(0,t)d_tA(0,t)+\int_{[\mathbf{0},\mathbf{T}]}W(\mathbf{s})dA(\mathbf{s})
 +\int_{[\mathbf{0},\mathbf{T}]}d_1W(\mathbf{t})d_2A(\mathbf{t}).$$
 From three   equalities above we  immediately get that
\BQNY
\int_{[\mathbf{0},\mathbf{T}]}A(\mathbf{s})dW(\mathbf{s})
 &=&\int_{[0,T]}A(s,T)d_sW(s,T)-\int_{[\mathbf{0},\mathbf{T}]}d_1W(\mathbf{t})d_2A(\mathbf{t})-\int_{[0,T]}A(s,0)dW_1(s)\\
 &=&\int_{[0,T]}A(s,T)d_sW(s,T)-\int_{[0,T]}W(T,t)d_tA(T,t)\\
 &+&
 \int_{[\mathbf{0},\mathbf{T}]}W(\mathbf{s})dA(\mathbf{s})+\int_{[0,T]}W(0,t)d_tA(0,t)-\int_{[0,T]}A(s,0)dW_1(s)\\
 &=&
A(\mathbf{T})W(\mathbf{T})-A(T,0)W(T,0)-A(0,T)W(0,T)+\int_{[\mathbf{0},\mathbf{T}]}W(\mathbf{s})dA(\mathbf{s})\\
&&+\int_{[0,T]}W(s,T)d_s(-A(s,T))
+\int_{[0,T]}W(T,t)d_t(-A(T,t))\\&&+\int_{[0,T]}W_1(s)d_s(A(s,0))+\int_{[0,T]}W_2(t)d_s(A(0,t))
\EQNY
establishing the proof.
\end{proof}

\COM{
\begin{lem}\label{lem2.1} If the function $h:\R_+^2\rightarrow \R$ admits the representation
\begin{equation}\label{eqhm-2}
 h(\mathbf{t})=h_1(t_1)+h_2(t_2)+ h_3(\mathbf{t}),
\end{equation}
 where $h_i\in \kHA, i=1,2$ and $h_3\in \kHB.$ Then \pE{the} representation \eqref{eqhm-2} is unique.
\end{lem}
The proof follows immediately if we put $t_i=0, i=1,2.$

Note that the covariance function of the field $W$ from \eqref{eqhm-1} equals
\begin{equation}\label{eqhm-3}
\EE{W(\mathbf{s})W(\mathbf{t})}=s_1\wedge t_1+s_2\wedge t_2+(s_1\wedge t_1)(s_2\wedge t_2).
\end{equation}
Now, introduce the space
\begin{equation*}
\kHC=\{h:\R_+^2\rightarrow \R\big|h(\mathbf{t})=\sum_{i=1,2}h_i(t_i)+h_3(\mathbf{t}),\; \text{where}\;  h_i\in \kHA,  i=1,2 \; \text{and}\; h_3\in \kHB \}.
\end{equation*}
The next result follows immediately from the definition of RKHS, Theorem 5, P. 24 \cite{berlinet} and Lemma \ref{lem2.1}.
}
\begin{lem} \label{lem:kHC}
The RKHS related to covariance function of the process $W$ coincides with $\kHC$ given in \eqref{eqhm-9}.
\end{lem}

\begin{proof} If the function $h:\R_+^2\rightarrow \R$ admits the representation
\begin{equation}\label{eqhm-2}
 h(\mathbf{t})=\sum_{i=1,2}h_i(t_i)+h_3(\mathbf{t}),
\end{equation}
 where $h_i\in \kHA, i=1,2$ and $h_3\in \kHB,$ then the representation \eqref{eqhm-2} is unique. This claim
 follows immediately if we put $t_i=0, i=1,2.$  In view of \eqref{eqhm-3}
the claim follows by Theorem 5, p.24 in \cite{berlinet}.
\end{proof}

Consider the subspace $V_1=\{h\in \kHA  \big|\,\Delta_s^1 h'(t)\leq 0\;, 0\leq s\leq t<+\infty\}.$
  Evidently, for any $h\in V_1$ we have that $h'(t)\downarrow 0$ as $t\rightarrow\infty.$
   Now we establish similar fact for the subspace
    $$V_2=\{h \in \kHB \big|\Delta_\mathbf{s}h''(\mathbf{t})\geq 0,\;
    \Delta^1_\mathbf{s}h''(t_1,s_2)\leq 0,\;\Delta^2_\mathbf{s}h''(s_1,t_2)\leq 0\;\text{for any}\;
     \mathbf{s}\leq\mathbf{t}\; \text{and}\;\mathbf{t}\in \R^2_+ \}.$$
\begin{lem} \label{lem5} Let function $h \in V_2$ and suppose that $\int_{\R_+}(h''(s,0))^2ds<\infty$ and $\int_{\R_+}(h''(0,t))^2dt<\infty$. Then  $h''(s,t)\downarrow 0$ as $s\rightarrow\infty$ for any $t\in R_+$,  $h''(\pE{s},t)\downarrow 0$ as $t\rightarrow\infty$ for any $s\in R_+$, and $h''(s,t)\downarrow 0$ as $s,t\rightarrow\infty.$
\end{lem}

\begin{proof} It is sufficient to establish the first formula. We have from $\int_{\R^2_+}(h''(s,t))^2dsdt<\infty$ that
$\int_{\R_+}(h''(s,t))^2ds<\infty$ for a.e. $t$. Furthermore, $h''(s,t)$ is non-increasing in  $s$ therefore for such $t$ we have
$h''(s,t)\downarrow 0$ as $s\rightarrow\infty$ and it follows from the assumption that $h''(s,0)\downarrow 0$ as $s\rightarrow\infty$. Since it is non-increasing in  $t$, we get such convergence for any $t$, hence the claim follows.
\end{proof}

{\bf Acknowledgment}:  Both authors were partially supported by the Swiss National Science Foundation project 200021-140633/1 and the project RARE -318984
 (a Marie Curie FP7 IRSES Fellowship).

\bibliographystyle{plain}

 \bibliography{GYYYZ}

\end{document}